\documentclass[preprint,3p]{elsarticle}
\usepackage{graphicx,color}
\usepackage{amsmath,amsfonts,bbm,mathtools}
\usepackage{comment}
\newtheorem{theorem}{Theorem}
\newtheorem{lemma}[theorem]{Lemma}

\newdefinition{definition}{Definition}
\newdefinition{assumption}{Assumption}
\newdefinition{remark}{Remark}
\newproof{proof}{Proof}

\providecommand{\citep}[1]{(\cite{#1})}
\newcommand{\pr}[1]{(#1(t), \, 0\leq t\leq 1)}

\newcommand{\cA}{\mathcal{A}}

\newcommand{\cF}{\mathcal{F}}
\newcommand{\cG}{\mathcal{G}}

\newcommand{\bF}{\mathbb{F}}
\newcommand{\bG}{\mathbb{G}}

\newcommand{\bOne}{\mathbbm{1}}
\newcommand{\bR}{\mathbb{R}}
\newcommand{\bV}{\mathbb{V}}
\newcommand{\PP}{\boldsymbol{P}}
\newcommand{\EE}{\boldsymbol{E}}

\newcommand{\defeq}{\coloneqq}
\providecommand{\keywords}[1]
{
  \small	
  \textbf{\textit{Keywords---}} #1
}

\begin{document}

\begin{frontmatter}

\title{A short note on ``Anticipative portfolio optimization''}

\author[uc3m,nyuad]{B.~D'Auria\fnref{fn1}\corref{cor1}}
\ead{bernardo.dauria@uc3m.es}
\author[uc3m]{J.-A.~Salmer\'on\fnref{fn1}}
\ead{joseantonio.salmeron@uc3m.es}

\cortext[cor1]{Corresponding author}
\fntext[fn1]{This research was partially supported by the Spanish Ministry of Economy and Competitiveness Grants  MTM2017-85618-P (via FEDER funds) and MTM2015--72907--EXP.}

\address[uc3m]{Dpto. Estad\'{\i}stica, Universidad Carlos III de Madrid.
Avda.\ de la Universidad 30, 28911, Legan\'es (Madrid) Spain}
\address[nyuad]{Science division, New York University of Abu Dhabi,
Saadiyat Island, P.O. box 129188, Abu Dhabi, UAE}
\begin{abstract}
    In \cite{karatzas2}, Pikovsky and Karatzas did one of the earliest studies on portfolio optimization problems in presence of insider information. They were able to successfully show that the knowledge of the stock price at future time is an insider information with associated unbounded value. However when the insider information only gives an interval containing the future value of the stock price, they could not prove that the value of the information is finite. They made a conjecture of this result and tried to convince about its validity by showing some numerical approximations. Instead of applying more sophisticated results, such as Shannon’s information theory, we show that their conjecture holds true by giving a simple proof of the finiteness of the value of the insider information for this case.\\
    
\end{abstract}

\end{frontmatter}
\section*{Introduction}
The problem of the optimal portfolio involves choosing the allocation of a certain capital among certain assets in order to maximize the expectation of profits in a finite time horizon. It is one of the most commonly optimization problems used in economics and finance. By assets, we refer to stocks of listed companies and bank bonds under a deterministic or stochastic interest short rate.

Sometimes it can be the case that the information available to the agent to define her strategy is greater than the one generated by the historical data of the market. 
In these cases, we say that the investor has additional or privileged information, that may include knowledge about future market behavior. Knowing this information leads to an increase in the expected gains of the investor. 
Pikovsky and Karatzas enunciate and solve some basic problems of insider trading by using enlargement of filtration techniques. 
First, the authors assume that the price of a stock is known at a future time and prove that, in this context, the agent has infinite expected wealth. 
After that, they weaken the information owned by the agent, and assume that she only knows a lower (or upper) bound of the future value of the stock price and they show that her profits are limited. 
Finally they show an example where the agent knows a finite interval that contains the future value of the stock price and state that also for this case  the expected profits remain finite. 
However, they are not able to provide an analytic proof of this result and use simulations to conclude that it should hold true. 
In this short note, we give a technical proof of that result, so showing their Conjecture 4.9 in \cite{karatzas2} holds true.
\begin{remark}
Another proof of this result can be found in \cite{Amendinger98}, where the more general framework of Entropy theory is presented.  Karatzas and Pikovsky, as indicated in \citep[Remark (2.5)]{karatzas2}, were already aware of possible relationships existing between the price of the information and the relative entropy. However they intended to solve their conjeture by simple analytical techniques and without using more powerful results. By this work we fulfilled this attempt.
\end{remark}
\section*{Notation and Main Results}
Trying to strictly follow the notation in \cite{karatzas2},
we assume to work in a filtered probability space
$(\Omega, \cF,\bF,\PP)$
where $\cF$ is the event sigma-algebra, and $\bF=\{\cF_t, 0\leq t\leq 1\}$ is an augmented filtration that is generated by the natural filtration of a Brownian motion $w = (w(t), 0\leq t \leq 1)$.
The portfolio is made of only two assets, one risky, that we call $P_1=\pr{{P_1}}$ and the other risk-less $P_0=\pr{{P_0}}$,
and both processes are considered adapted semi-martingales in the the filtration $\bF$.
In particular, their dynamics are defined by the following SDEs,
\begin{subequations}\label{DS.SDE}
\begin{align}
dP_0(t) &= P_0(t) \, r(t) \, dt, 
\label{D.SDE}\\
dP_1(t) &= P_1(t)  \left(b(t) dt + \sigma(t) \, dw(t)\right)   \label{S.SDE}
\end{align}
\end{subequations}
where $r$ is the interest rate, $b$ is the drift and $\sigma$ is the volatility of the risky asset. 
They are assumed to be bounded deterministic function. In particular, we remark that there exist $s_1,\,s_2\in\bR^+$ 
such that $0<s_1\leq \sigma(u)\leq s_2$ 
holds true for all $0\leq u \leq 1$ \citep[Equation (4.23)]{karatzas2}.

With the previous set-up, we assume that an agent can control her portfolio by a given self-financial strategy $\pi=\pr{\pi}$,
in order to optimize her utility function at a finite terminal time.
If we denote by $X^\pi=\pr{X^\pi}$ the wealth of the portfolio of the investor under her strategy $\pi$,
its dynamics are given by the following SDE, for $0 \leq t \leq 1$,

\begin{align}\label{X.SDE}
dX^{\pi}(t) &= (1-\pi(t))X^{\pi}(t)r(t) \, dt + \pi(t) X^{\pi}(t)\left(b(t) \, dt + \sigma(t) \, dw(t)\right) \ ,
\quad X(0) = x \ .
\end{align}

Usually, it is assumed that the optimal strategy uses all the information at disposal of the agent at each instant,
and in general we assume that the agent's flow of information, modeled by the filtration $\bG=\pr{\cG}$,
is possibly larger than filtration $\bF$, that is  $\bF \subset \bG$. In particular, we consider that the extra information is given by a random variable $\mathbf{L}$ that is $\cF_1$-measurable and  define the following initial enlargement of filtration
$$ \cG_t = \cF_t\bigvee\sigma(\mathbf{L}),\ \ 0\leq t\leq 1. $$

In this new filtration the process $w$ is not a $\bG$-Brownian motion, therefore it is necessary to compute its semi-martingale decomposition with respect to a $\bG$-Brownian motion before solving the optimization problem. 
In \cite[Lemma 2.3]{karatzas2} the following result is proved, that we give here adapted for the special case of a single risky asset.
\begin{lemma}[Characterization Lemma]
Assume that for the given $\cF_1$-measurable random variable $\mathbf{L}$, we can find a measurable process $\alpha:[0,1)\times\Omega\longrightarrow \bR$, such that:
\begin{enumerate}
    \item $\alpha$ is adapted to $\bG = \{ \cG_t \}_{0\leq t\leq 1}$, with $\cG_t = \cF_t \bigvee\sigma(\mathbf{L})$.
    \item the process $\tilde w(t) = w(t) -\int_0^t \alpha(u) \, du$ is a $\bG$-BM on $[0,1)$.
    \item $\EE \int_0^T \alpha^2(s) ds < +\infty$ for any $T<1$.
\end{enumerate}
For $T\in (0,1]$, let $\mathcal{A}(\bG, T)$ (resp. $\mathcal{A}(\bF, T)$) be the class of $\bG$ (resp. $\bF$)-adapted processes 
$\pi:[0,T]\times \Omega\longrightarrow\bR$ with $\int_0^T (\sigma(t) \pi(t))^2 dt < +\infty$, almost surely, and let
\begin{align}
    \bV_T^{\bF} &\defeq \sup_{\pi\in\cA_{\bF}} \EE[\log(X^{\pi}_{T})]\ ,\\
    \bV_T^{\bG} &\defeq \sup_{\pi\in\cA_{\bG}} \EE[\log(X^{\pi}_{T})]\ ,
\end{align}
denote the values of the portfolio optimization problem over these two respective classes (i.e. with or without anticipation of the terminal value $\mathbf{L}$, respectively). Then
\begin{align}
    \bV_T^{\bF} &= \log(x) + \EE\int_0^T \left( r(u) + \frac{1}{2} \left(\frac{b(u)-r(u)}{\sigma(u)}\right)^2 \right) \, du \ ,\\
    \bV_T^{\bG} &=  \bV_T^{\bF} + \frac{1}{2}\EE\int_0^T \alpha^2(u) \, du,\ \ 0<T\leq 1.
\end{align}
and thus
\begin{equation}
    V_1^{\bG} < \infty \Longleftrightarrow \EE\int_0^1 \alpha^2(u) \, du < \infty \ . \label{finite.VG}
\end{equation}
When this latter condition is satisfied, an optimal portfolio is given by
\begin{equation}
    \pi^*(t) = \frac{b(t)-r(t)}{\sigma^2(t)} + \frac{\alpha(t)}{\sigma(t)} \ .
\end{equation}
\end{lemma}

\section*{Finiteness of the value of the insider information for a finite interval.}
We focus, in this note, on the case when the price is one-dimensional and the extra information is given by the following random variable,
$$\mathbf{L} =  \bOne \{ P_1(1)\in [p_1,p_2)\},$$
which gives us an upper and a lower bound for the stock price at the end of the time horizon. This is the same setting as in \cite[Example 4.8]{karatzas2}.

We can solve easily the equation related to the asset $P_1$ and get
$$ P_1(1) = P_1(0)\exp \left( \int_0^1 (b(u)-\sigma^2(u)/2) \, du + \int_0^1 \sigma(u) \, dw(u) \right), $$
so that we can express the variable $\mathbf{L}$ in terms of the process $w$
as $\mathbf{L} =  \bOne \{ \int_0^1\sigma(u) \, dw(u) \in[c_1,c_2)\}$ where, for $i \in {1,2}$, we define
$$ c_i = \log(p_i/P(0)) + \int_0^1 \left( \frac{\sigma^2(u)}{2}-b(u)  \right) \, du \ .$$
The technical proof gives the crucial ingredient for the proof of the main result.
\begin{lemma}\label{lem:bound.I.finite}
The integral of the function $I(x,t)$ defined as 
\begin{equation}
    I(x,t) = \frac{ \left[ \Phi'(z_1)-\Phi'(z_2)\right]^2 }{\sqrt{\rho-\tau}\left[\Phi\left( z_2  \right)  - \Phi\left( z_1  \right)\right] \left[\Phi\left( -z_2  \right)  + \Phi\left( z_1  \right)\right]}, \label{def.I}
\end{equation}
in the variable $x\in\bR$ is uniformly bounded for $t\in[0,1]$.
In \eqref{def.I}, we have used the following definitions
$\rho = T(1)$, $\tau = T(t)$, with $T(t) = \int_0^t\sigma^2(u) \, du$, $z_2 = {(c_2-x)}/{\sqrt{\rho-\tau}}$ and 
$z_1 = {(c_1-x)}/{\sqrt{\rho-\tau}}$. 
\end{lemma}
\begin{proof}
We start by splitting $\bR$ in three intervals  $(-\infty,c_1]$, $(c_1,c_2)$ and $[c_2,\infty)$, then we prove that on each interval the integral is finite.

\textbf{Interval $(-\infty,c_1]$:}
We apply a change of variable in $z_1$ and express $z_2 = z_1 + (c_2-c_1)/\sqrt{\rho-\tau}$. 
We let $s_t:=(c_2-c_1)/\sqrt{\rho-\tau}$ and call its minimum in $t$ as $s_0>0$. 
We get
\begin{align}
\int_{-\infty}^{c_1}I(x,t)dx  &=  \int_0^{+\infty} \frac{\left[\Phi'(z_1)-\Phi'(z_1 + s_t)\right]^2  }{\left[\Phi\left( z_1 +s_t  \right)  - \Phi\left( z_1  \right)\right] \left[\Phi\left( -z_1 - s_t  \right)  + \Phi\left( z_1  \right)\right]} dz_1 \nonumber \\
&= \int_0^{+\infty} \left( \frac{ \left[\Phi'(z_1)-\Phi'(z_1 + s_t)\right]^2}{\Phi\left( z_1 +s_t  \right)  - \Phi\left( z_1  \right)} + \frac{\left[\Phi'(z_1)-\Phi'(z_1 + s_t)\right]^2}{\Phi\left( -z_1 - s_t  \right)  + \Phi\left( z_1  \right)}  \right) dz_1 \ .
\label{lower.semi.interval}\end{align}
We continue by showing that both terms are finite.
We first consider the first term.
\begin{align*}
\int_0^{+\infty}\frac{ \left[\Phi'(z_1)-\Phi'(z_1 + s_t)\right]^2}{\Phi\left( z_1 +s_t  \right)  - \Phi(z_1)} dz_1 &\leq \int_0^{+\infty}\frac{ \left[\Phi'(z_1)-\Phi'(z_1 + s_t)\right]^2}{\Phi\left( z_1 +s_0  \right)  - \Phi(z_1)} dz_1 \leq \int_0^{+\infty}\frac{ \left[\Phi'(z_1)\right]^2}{\Phi\left( z_1 +s_0  \right)  - \Phi(z_1)} dz_1\\
&= \int_0^{1}\frac{ \left[\Phi'(z_1)\right]^2}{\Phi\left( z_1 +s_0  \right)  - \Phi(z_1)} dz_1 + \int_1^{+\infty}\frac{ \left[\Phi'(z_1)\right]^2}{\Phi\left( z_1 +s_0  \right)  - \Phi(z_1)} dz_1 \ .
\end{align*}
The first integral is clearly bounded while, for the second one, we apply a comparison criteria with the function $f(z) = 1/z^2$, as follows
\begin{align*}
\lim_{z_1\to\infty} z_1^2 \frac{ \left[\Phi'(z_1)\right]^2}{\Phi\left( z_1 +s_0  \right)  - \Phi(z_1)} &= \lim_{z_1\to\infty} \frac{1}{\sqrt[]{2\pi}} \frac{z_1^2\left[\exp\left( -{z_1^2}/{2} \right) \right]^2}{\int_{z_1}^{z_1+s_0}\exp(-u^2/2) \, du}\\
&=\lim_{z_1\to\infty} \frac{1}{\sqrt[]{2\pi}} \frac{2z_1\left[\exp\left( -{z_1^2}/{2} \right) \right]^2 - 2z_1^3 \left[\exp\left( -{z_1^2}/{2} \right) \right]^2}{\exp(-(z_1+s_0)^2/2) - \exp(-z_1^2/2)}\\
&= \lim_{z_1\to\infty} \frac{1}{\sqrt[]{2\pi}} \frac{2z_1\exp\left( -{z_1^2}/{2} \right) - 2z_1^3 \exp\left( -{z_1^2}/{2} \right) }{\exp(-z_1s_0)\exp(-s_0^2/2) - 1}\to 0 \ .
\end{align*}
In the second equality above, we used L'Hopital Rule and we conclude that the integral is finite on $(1,+\infty)$. 

As for the second term in \eqref{lower.semi.interval}, we have the following bound,
\begin{align*}
\int_0^{+\infty} \frac{\left[\Phi'(z_1)-\Phi'(z_1 + s_t)\right]^2}{\Phi(-z_1-s_t)  + \Phi(z_1)} dz_1 &\leq \int_0^{+\infty} \frac{\left[\Phi'(z_1)-\Phi'(z_1 + s_t)\right]^2}{\Phi(z_1)} dz_1 \leq 2 \int_0^{+\infty}\left[\Phi'(z_1)-\Phi'(z_1 + s_t)\right]^2dz_1\\
&\leq 2 \int_0^{+\infty}\left[\Phi'(z_1)\right]^2dz_1 = \frac{1}{\sqrt[]{2}} \ .
\end{align*}

\textbf{Interval $[c_2,+\infty)$:}
We proceed in the same way as above, but now applying a change of variable in $z_2$.
\begin{align}
\int^{+\infty}_{c_2}I(x,t)dx &= \int_{-\infty}^{0}\frac{ \left[ \Phi'(z_2-s_t)-\Phi'(z_2)\right]^2  }{\left[\Phi\left( z_2  \right)  - \Phi\left( z_2-s_t  \right)\right] \left[\Phi\left( -z_2  \right)  + \Phi\left( z_2-s_t \right)\right]} dz_2 \nonumber \\
&= \int_{-\infty}^{0} \left(\frac{ \left[ \Phi'(z_2-s_t)-\Phi'(z_2)\right]^2  }{\Phi\left( z_2  \right)  - \Phi\left( z_2-s_t  \right) }+ \frac{ \left[ \Phi'(z_2-s_t)-\Phi'(z_2)\right]^2  }{\Phi\left( -z_2  \right)  + \Phi\left( z_2-s_t \right)}\right) dz_2
\label{upper.semi.interval}\end{align}
We show that both terms in \eqref{upper.semi.interval} are finite. 
For the first one we have
\begin{align*}
\int_{-\infty}^{0}\frac{ \left[ \Phi'(z_2-s_t)-\Phi'(z_2)\right]^2  }{\Phi\left( z_2  \right)  - \Phi\left( z_2-s_t  \right) } dz_2 \leq \int_{-\infty}^{0}\frac{ \left[ \Phi'(z_2-s_t)-\Phi'(z_2)\right]^2  }{\Phi\left( z_2  \right)  - \Phi\left( z_2-s_0  \right) } dz_2 \leq \int_{-\infty}^{0}\frac{ \left[\Phi'(z_2)\right]^2  }{\Phi\left( z_2  \right)  - \Phi\left( z_2-s_0  \right) } dz_2 \ ,
\end{align*}
and applying the same reasoning as before, we conclude that the integral is finite. Then for the second term we have
\begin{align*}
 \int_{-\infty}^{0}\frac{ \left[ \Phi'(z_2-s_t)-\Phi'(z_2)\right]^2  }{\Phi\left( -z_2  \right)  + \Phi\left( z_2-s_t \right)} dz_2 \leq \int_{-\infty}^{0}\frac{ \left[ \Phi'(z_2-s_t)-\Phi'(z_2)\right]^2  }{\Phi\left( -z_2  \right)  } dz_2 \leq 2\int_{-\infty}^{0} \left[\Phi'(z_2)\right]^2 dz_2 = \frac{1}{\sqrt[]{2}} \ .
\end{align*}

\textbf{Interval $(c_1,c_2)$:}
We proceed by applying a change of variable, and we arbitrarily choose to do it in the variable $z_2$. 
We get
\begin{align}
\int_{c_1}^{c_2}I(x,t)dx &= \int_0^{s_t} \frac{[\Phi'(z_2)-\Phi'(z_2-s_t)]^2}{\Phi(z_2) - \Phi(z_2-s_t) }+\frac{[\Phi'(z_2)-\Phi'(z_2-s_t)]^2}{\Phi(-z_2) + \Phi(z_2-s_t) } dz_2
\label{middle.interval}\end{align}
and again we show that both integrals in \eqref{middle.interval} are bounded.
For the first integral we have
\begin{align*}
\int_0^{s_t} \frac{[\Phi'(z_2)-\Phi'(z_2-s_t)]^2}{\Phi(z_2) - \Phi(z_2-s_t) }dz_2&= 2\int_0^{s_t/2} \frac{[\Phi'(z_2)-\Phi'(z_2-s_t)]^2}{\Phi(z_2) - \Phi(z_2-s_t) }dz_2\leq 2\int_0^{s_t/2} \frac{[\Phi'(z_2)]^2}{\Phi(z_2) - \Phi(z_2-s_0) }dz_2\\
&\leq 2\int_0^{\infty} \frac{[\Phi'(z_2)]^2}{\Phi(z_2) - \Phi(z_2-s_0) }dz_2
\end{align*}
where the first equality holds because the function we are integrating is symmetric with respect $s_t/2$. 
The last integral is finite as it is trivially so on $[0,1]$ and using a comparison criteria with the function $f(z)=1/z^2$ 
it is also integrable on $[1,+\infty]$. 
In a similar way we analyze the second integral in \eqref{middle.interval} and by symmetry of the function with respect to the $s_t/2$ we get 
\begin{align*}
\int_0^{s_t}\frac{[\Phi'(z_2)-\Phi'(z_2-s_t)]^2}{\Phi(-z_2) + \Phi(z_2-s_t) } dz_2& =  2\int_0^{s_t/2} \frac{[\Phi'(z_2)-\Phi'(z_2-s_t)]^2}{\Phi(-z_2) + \Phi(z_2-s_t) }dz_2 \ .
\end{align*}
Then we compute the following bound
\begin{align*}
\Phi'(z_2)-\Phi'(z_2-s_t) &= \frac{1}{\sqrt[]{2\pi}}\left[ \exp\left(-\frac{z_2^2}{2}\right) -\exp\left(-\frac{(z_2-s_t)^2}{2}\right)\right] =\frac{1}{\sqrt[]{2\pi}}\exp\left(-\frac{z_2^2}{2}\right)\left[1-\exp\left(-\frac{s_t^2-2z_2s_t}{2}\right)\right]\\
&\leq \frac{1}{\sqrt[]{2\pi}}\exp\left(-\frac{z_2^2}{2}\right) \ ,
\end{align*}
where the  last inequality holds because $s_t^2-2z_2s_t\geq 0$ as $0\leq z_2\leq s_t/2$.
\begin{align*}
\int_0^{s_t}\frac{[\Phi'(z_2)-\Phi'(z_2-s_t)]^2}{\Phi(-z_2) + \Phi(z_2-s_t) } dz_2 &\leq \sqrt[]{\frac{2}{\pi}} \int_0^{s_t/2} \frac{\exp\left(-z_2^2\right)}{\Phi(-z_2) + \Phi(z_2-s_t) }dz_2 \leq  \int_0^{s_t/2} \frac{\exp\left(-z_2^2\right)}{\Phi(-z_2) }dz_2 \\
&\leq \int_0^1 \frac{\exp\left(-z_2^2\right)}{\Phi(-z_2) }dz_2 + \int_1^{+\infty} \frac{\exp\left(-z_2^2\right)}{\Phi(-z_2) }dz_2 \ .
\end{align*}
The first integral is trivially bounded. For the second to be bounded, we apply a comparison criteria with the function $f(z) = 1/z^2$. Putting together the given bounds we may bound the integral in \eqref{middle.interval} and the proof is finished.
\end{proof}

We can now state the main result that solves the Conjecture 4.9 of \cite{karatzas2}.

\begin{theorem}
When $\mathbf{L} =  \bOne \{ P_1(1)\in [p_1,p_2)\}$, the value of the insider information is finite, that is $V_1^{\bG} < \infty$.
\end{theorem}
\begin{proof}
By using the expression of $\alpha$, given in \cite[Equation (4.25)]{karatzas2}, we have
$$ \EE[\alpha^2(t)] =  \frac{\sigma^2(t)}{2\pi\sqrt{\rho-\tau}\sqrt{2\pi\tau}} \int_{\mathbb{R}}  I(x,t) \, e^{-x^2/2} \, dx $$
where $I(x,t)$ is defined in \eqref{def.I}.
By \eqref{finite.VG}, it is enough to prove that, for some constant $K>0$,
$$ \EE[\alpha^2(t)] \leq \frac{K}{\sqrt{t(1-t)}} \ . $$
This follows by applying Lemma \ref{lem:bound.I.finite} and noticing that, by the standing assumption that $0<s_1\leq \sigma(u)\leq s_2$ for some $s_1,\, s_2 \in \bR^+$ and $u\in[0,1]$, we have that 
$$ s_1(1-t) < \rho - \tau = \int_t^1 \sigma^2(u) \, du < s_2(1-t) \ .$$
\end{proof}

\end{document}